\documentclass[12pt,reqno]{amsart}

\newcommand\version{June 16, 2015}

\usepackage{amsmath,amsfonts,amsthm,amssymb,amsxtra}

\newtheorem{theorem}{Theorem}[section]
\newtheorem{proposition}[theorem]{Proposition}
\newtheorem{lemma}[theorem]{Lemma}
\newtheorem{corollary}[theorem]{Corollary}

\theoremstyle{definition}

\theoremstyle{remark}

\newtheorem{remark}[theorem]{Remark}

\numberwithin{equation}{section}

\newcommand{\C}{\mathbb{C}}

\renewcommand{\epsilon}{\varepsilon}

\newcommand{\N}{\mathbb{N}}

\renewcommand{\phi}{\varphi}
\newcommand{\R}{\mathbb{R}}

\newcommand{\Sph}{\mathbb{S}}

\newcommand{\Z}{\mathbb{Z}}

\DeclareMathOperator{\im}{Im}

\DeclareMathOperator{\sech}{sech}

\DeclareMathOperator{\sgn}{sgn}

\begin{document}

\title[Schr\"odinger operators with complex potentials --- \version]{Eigenvalue bounds for Schr\"odinger operators\\ with complex potentials. II}

\author{Rupert L. Frank}
\address{Rupert L. Frank, Mathematics 253-37, Caltech, Pasadena, CA 91125, USA}
\email{rlfrank@caltech.edu}

\author{Barry Simon}
\address{Barry Simon, Mathematics 253-37, Caltech, Pasadena, CA 91125, USA}
\email{bsimon@caltech.edu}

\begin{abstract}
Laptev and Safronov conjectured that any non-positive eigenvalue of a Schr\"odinger operator $-\Delta+V$ in $L^2(\R^\nu)$ with complex potential has absolute value at most a constant times $\|V\|_{\gamma+\nu/2}^{(\gamma+\nu/2)/\gamma}$ for $0<\gamma\leq\nu/2$ in dimension $\nu\geq 2$. We prove this conjecture for radial potentials if $0<\gamma<\nu/2$ and we `almost disprove' it for general potentials if $1/2<\gamma<\nu/2$. In addition, we prove various bounds that hold, in particular, for positive eigenvalues.
\end{abstract}

\renewcommand{\thefootnote}{${}$} \footnotetext{\copyright\, 2015 by the authors. This paper may be reproduced, in its entirety, for non-commercial purposes.\\
Work partially supported by U.S. National Science Foundation grants PHY-1347399, DMS-1363432 (R.L.F.) and DMS-1265592 (B.S.).}

\maketitle

\section{Introduction and main results}

In this paper we are interested in eigenvalues of Schr\"odinger operators
$$
-\Delta +V \qquad\text{in}\ L^2(\R^\nu)
$$
with (possibly) complex-valued potentials $V$. More precisely, we want to derive bounds on the location of these eigenvalues assuming only that $V$ belongs to some $L^p(\R^\nu)$ with $p<\infty$. This assumption, for suitable $p$, will also guarantee that $-\Delta+V$ can be defined via the theory of $m$-sectorial forms. Also, $p<\infty$ implies that eigenvalues outside of $[0,\infty)$ are discrete and have finite algebraic multiplicities.

If $V$ is real-valued (so that discrete eigenvalues are negative), it is a straightforward consequence of Sobolev inequalities that
\begin{equation}
\label{eq:keller}
|E|^\gamma \leq C_{\gamma,\nu} \int_{\R^\nu} |V|^{\gamma+\nu/2}\,dx
\end{equation}
for every $\gamma\geq 1/2$ if $\nu=1$ and every $\gamma>0$ if $\nu\geq 2$. Here $C_{\gamma,\nu}$ is a constant independent of $V$. For this bound, see \cite{K,LT} and also \cite{CFL} for optimal constants, optimal potentials and stability results.

The question becomes much more difficult if $V$ is allowed to be complex-valued. Laptev and Safronov \cite{LS} conjectured that for any $\nu\geq 2$ and $0<\gamma\leq\nu/2$ there is a $C_{\gamma,\nu}$ such that \eqref{eq:keller} holds for all eigenvalues $E\in\C\setminus[0,\infty)$. Prior to their conjecture, Abramov, Aslanyan and Davies \cite{AAD} (see also \cite{DN}) had shown this for $\nu=1$ and $\gamma=1/2$. In \cite{F} the Laptev--Safronov conjecture was proved for $\nu\geq 2$ and $0<\gamma\leq1/2$.

\medskip

In this paper we accomplish the following:
\begin{enumerate}
\item[(A)] We almost disprove the Laptev--Safronov conjecture for $\nu\geq 2$ and $1/2<\gamma<\nu/2$ (Theorem \ref{ij}).
\item[(B)] We prove the Laptev--Safronov conjecture for \emph{radial} potentials for $\nu\geq 2$ and $1/2<\gamma<\nu/2$.
\item[(C)] We give a simple proof that for $0<\gamma\leq 1/2$ the bound \eqref{eq:keller} holds also for eigenvalues $E\in [0,\infty)$. (We note that a deep result of Koch--Tataru \cite{KT} shows that, in fact, there are no positive eigenvalues.)
\item[(D)] We prove an eigenvalue bound for $V\in L^{\gamma_1+\nu/2}(\R^\nu) + L^{\gamma_2+\nu/2}(\R^\nu)$ with $0<\gamma_1<\gamma_2\leq 1/2$ if $\nu=2$ and $0\leq \gamma_1<\gamma_2\leq 1/2$ if $\nu\geq 3$.
\end{enumerate}
By `almost disprove' in (A) we mean we construct a sequence of real-valued potentials $V_n$ such that $-\Delta +V_n$ has eigenvalue $1$ but $\|V_n\|_p\to 0$ for any $p>(1+\nu)/2$. If Laptev and Safronov had formulated their conjecture for any eigenvalue $E\in\C$ (and not only for $E\in\C\setminus[0,\infty)$), we would have disproved it. In particular, this is interesting in view of (C), where we prove that for $0<\gamma\leq1/2$ the conjecture holds in fact also for eigenvalues in $[0,\infty)$. Note that if we were able to show that the eigenvalue $1$ of $-\Delta+V_n$ becomes a non-real eigenvalue of $-\Delta+V_n+\epsilon W$ for some nice $W$ (say with $\im W\geq 0$) and $\epsilon$ small, we could also disprove the conjecture.

Our construction of the potentials $V_n$ in the proof of Theorem \ref{ij} is inspired by a construction of Ionescu and Jerison \cite{IJ}. Using ideas of Wigner and von Neumann \cite{WvN} (see also \cite[Section XIII.13]{RS4}) we are able to simplify their construction. 

We also prove (Theorem \ref{wvn}) that a bound of the form \eqref{eq:keller} cannot hold, even for radial potentials, if $\gamma>\nu/2$. Of course, Laptev and Safronov conjectured such a bound only for $\gamma<\nu/2$, but the fact that this is the correct upper bound is not obvious. Our construction extends the Wigner--von Neumann construction \cite{WvN} (see also \cite{RS4}) to arbitrary dimension $\nu$, which is interesting in its own right. Our counterexamples are constructed in Section \ref{sec:negative}. In passing we mention that while the Wigner--von Neumann example has been studied extensively, we are not aware of similar results about the Ionescu--Jerison example. It would be interesting to extend the results of Naboko \cite{N} and Simon \cite{Si} on dense embedded point spectrum based on the Wigner--von Neumann example to instead use the Ionescu--Jerison example.

\medskip

Concerning (B), we recall that the proof in \cite{F} of \eqref{eq:keller} for $0<\gamma\leq 1/2$ relied on uniform Sobolev bounds due to Kenig--Ruiz--Sogge \cite{KRS}, namely,
\begin{equation}
\label{eq:keruso}
\| (-\Delta- z)^{-1} f\|_{p'} \leq C |z|^{-\nu/2+\nu/p-1} \|f\|_{p} \,,
\qquad 2\nu/(\nu+2)< p\leq 2(\nu+1)/(\nu+3) \,,
\end{equation}
with $C$ independent of $z$ and with $p'=p/(p-1)$. (In \cite{KRS} this bound is only proved for $\nu\geq 3$, but the same argument works for $\nu=2$ as well, see \cite{F}.) The range of exponents $2\nu/(\nu+2)<p\leq 2(\nu+1)/(\nu+3)$ in \eqref{eq:keruso} corresponds to $0<\gamma\leq 1/2$ in \eqref{eq:keller}. Bounds of the form \eqref{eq:keruso} cannot hold for exponents $2(\nu+1)/(\nu+3)<p<2\nu/(\nu+1)$ (corresponding to $1/2<\gamma<\nu/2$). However, as we shall show (Theorem \ref{resmixed}), they do hold if one replaces the space $L^p(\R^\nu)$ by $L^p(\R_+,r^{\nu-1}\,dr; L^2(\Sph^{\nu-1}))$ and similarly for $L^{p'}(\R^\nu)$. In fact, these bounds prove \eqref{eq:keller} not only for radial potentials, but for general potentials in $L^{\gamma+\nu/2}(\R_+,r^{\nu-1}\,dr; L^\infty(\Sph^{\nu-1}))$ with the obvious replacement on the right side; see Theorem \ref{evbound}. We also prove a Lorentz space result at the endpoint $\gamma=\nu/2$; see Theorem \ref{endpoint}.

Our results for $1/2<\gamma\leq\nu/2$ are based on arguments by Barcelo, Ruiz and Vega \cite{BRV} and, in particular, precise bounds on Bessel functions. This is further discussed in Section~\ref{sec:positive} and in the appendix.

\medskip

We prove (C) in Section \ref{sec:fagain}. Our argument is based on \eqref{eq:keruso}, like that in \cite{F}, but is more direct and avoids Birman--Schwinger operators. As we mentioned above, the deep results of Koch and Tataru \cite{KT} imply that $-\Delta+V$ has no positive eigenvalues if $V\in L^{\gamma+\nu/2}(\R^\nu)$ with $0<\gamma<1/2$; see also \cite{IJ} for the case $\gamma=0$ in dimensions $\nu\geq 3$. (The fact that the results of \cite{KT} apply also to complex-valued potentials is not emphasized there, but is clear from their proof strategy via Carleman inequalities. Also, the fact that $V\in L^{\gamma+\nu/2}(\R^\nu)$ satisfies Assumption A.2 in \cite{KT} for $\gamma$ as above can be easily verified using Sobolev embedding theorems; see, for instance, the proof of Lemma 3.5 in \cite{FP}.)

We include our proof of (C) since it is much simpler than the arguments in \cite{IJ,KT} and since the same reasoning will give the assertion in (B) for $E\in [0,\infty)$ where the results of \cite{KT} are not applicable.

\medskip

The bounds mentioned in (D), see Theorem \ref{lpsum}, are new, even for $E\in\C\setminus[0,\infty)$. They are also derived from \eqref{eq:keruso}. Somewhat related bound in $\nu=1$ are contained in \cite{DN}.

\medskip

In this paper we have only discussed bounds on single eigenvalues. The situation for sums of eigenvalues is less understood and we refer to \cite{FLLS,LS,BGK,DHK,FS} and references therein for results and open questions in this direction. Also, we emphasize that we work only under an $L^p$ condition on $V$. In contrast, results under exponential decay assumptions are classical (see, e.g., \cite{Na,Ma1,Ma2} and also \cite{St1,St2}) and extensions to sub-exponential decay were studied in a remarkable series of papers of Pavlov \cite{Pa1,Pa2,Pa3}. For results in the discrete, one-dimensional case we refer, for instance, to \cite{EgGo1,EgGo2}.

\subsection*{Acknowledgemnts}
The authors would like to thank L. Golinskii, H. Koch, A. Laptev, O. Safronov and D. Tataru for helpful corresondence.


\section{Counterexamples}\label{sec:negative}

The following theorem shows, in particular, that the bound \eqref{eq:keller} cannot be valid for positive eigenvalues of Schr\"odinger operators with real potentials if $\nu\geq 2$ and $\gamma>(\nu+1)/2$. Our proof simplifies the construction of potentials that appeared in \cite{IJ} in a different, but related context.

\begin{theorem}\label{ij}
For any $\nu\geq 2$ there is a sequence of potentials $V_n:\R^\nu\to\R$, $n\in\N$, such that $1$ is an eigenvalue of $-\Delta+V$ in $L^2(\R^\nu)$ and
$$
|V_n(x)| \leq \frac{C}{n + |x_1| + |x'|^2} \,,
\qquad x=(x_1,x')\in\R\times\R^{\nu-1} \,,
$$
with $C>0$ independent of $n$. In particular, for any $p>(\nu+1)/2$,
$$
\|V_n\|_{L^p} \to 0
\qquad\text{as}\ n\to\infty \,.
$$
\end{theorem}

\begin{proof}
We look for an eigenfunction of the form $\psi(x) = w(x) \sin x_1$. Then
$$
-\Delta\psi = \psi - 2(\partial_x w) \cos x_1 - (\Delta w) \sin x_1 \,,
$$
so the eigenvalue equation will be satisfied if we set
$$
V := 2 \frac{\partial_1 w}{w} \cot x_1 + \frac{\Delta w}{w} \,.
$$
We need to choose $w$ in such a way that $\psi\in L^2$ and that $V$ satisfies the required bounds. In particular, $\partial_1 w$ needs to vanish where $\sin x_1$ does. In order to achieve this, we set
$$
g(x_1) := 4 \int_0^{x_1} \sin^2 y \,dy = 2 x_1 - \sin (2x_1)
$$
and
$$
w_n(x) := \left( n^2 + g(x_1)^2 + |x'|^4\right)^{-\alpha} \,. 
$$
The potential $V_n$ is defined with $w_n$ in place of $w$. The parameter $n$ here is not necessarily an integer, but we do require later that $n\geq 1$. Finally, the parameter $\alpha$ will be chosen so that $w\in L^2(\R^\nu)$ (which implies $\psi\in L^2(\R^\nu)$). Note that
$$
\int_{\R^\nu} |w_n(x)|^2 \,dx = 2 |\Sph^{\nu-2}| \int_0^\infty (n^2 + g(x_1)^2)^{-2\alpha+(\nu-1)/2} \,dx_1 \int_0^\infty \frac{r^{\nu-2}\,dr}{(1+r^4)^{2\alpha}}
$$
is finite provided $\alpha>\nu/4$, which we assume in the following. We do not keep track of the dependence of our estimates on $\alpha$.

A quick computation shows that
$$
V_n = - \frac{4\alpha}{m_n} gg' \cot x_1 + \frac{4\alpha(\alpha+1)}{m_n^2} \left( g^2 (g')^2 + 4 |x'|^6 \right) - \frac{2\alpha}{m_n} \left( (g')^2 + gg'' + 2(\nu+1)|x'|^2 \right)
$$
with $m_n(x) := n^2 + g(x_1)^2 + |x'|^4$. Note that $g' \cot x_1= 4\sin x_1 \cos x_1$ is bounded. Moreover, $|g|, |x'|^2 \leq m_n^{1/2}$ and $|g'|, |g''|\leq C$, so
$$
|V_n| \leq C \left( m_n^{-1/2} + m_n^{-1} \right) \,.
$$
Using $n\geq 1$, we find $m_n^{-1} \leq n^{-1} m_n^{-1/2}\leq m_n^{-1/2}$, so $|V_n|\leq C m_n^{-1/2}$. This bound is equivalent to the one stated in the theorem.

Finally, we note that by scaling
$$
\int_{\R^\nu} |V_n|^p \,dx \leq C \int_{\R^\nu} \frac{dx}{(n + |x_1|+|x'|^2)^p} = n^{-p+ (\nu+1)/2} C \int_{\R^\nu} \frac{dx}{(1 + |x_1|+|x'|^2)^p} 
$$
For $p>(\nu+1)/2$, the right side tends to  zero since $(1+|x_1|+|x'|^2)^{-1}\in L^p$ in this case. This finishes the proof of the theorem.
\end{proof}

We emphasize that the eigenfunctions corresponding to the eigenvalue 1 of $-\Delta+V_n$ can have arbitrarily fast or slow (consistent with being square-integrable) algebraic decay in $|x_1|+|x'|^2$. We also note that (for fixed $n$) the potential $V_n$ has the asymptotic behavior
\begin{align*}
V_n(x) = & - \frac{16\alpha x_1 \sin^2 (2x_1)}{4 |x_1|^2+|x'|^4} + \frac{16\alpha(\alpha+1)|x'|^6}{(4 |x_1|^2+|x'|^4)^2} - \frac{4\alpha ( 4x_1 \cos(2x_1) + (\nu+1)|x'|^2)}{4 |x_1|^2+|x'|^4} \\
& + O((|x_1|+|x'|^2)^{-2})
\end{align*}
as $|x_1|+|x'|^2\to\infty$.

Our next theorem shows, in particular, that the bound \eqref{eq:keller} cannot be valid for positive eigenvalues of Schr\"odinger operators with real, radial potentials if $\nu\geq 1$ and $\gamma>1/2$. Our proof extends the Wigner--von Neumann construction \cite{WvN} (see also \cite{RS4}) to arbitrary dimensions $\nu\geq 1$.

\begin{theorem}\label{wvn}
For any $\nu\geq 1$ there is a sequence of radial potentials $V_n:\R^\nu\to\R$, $n\in\N$, such that $1$ is an eigenvalue of $-\Delta+V$ in $L^2(\R^\nu)$ and
$$
|V_n(x)| \leq \frac{C}{n + |x|} \,,
\qquad x\in\R^\nu \,,
$$
with $C>0$ independent of $n$. In particular, for any $p>\nu$,
$$
\|V_n\|_{L^p} \to 0
\qquad\text{as}\ n\to\infty \,.
$$
\end{theorem}

\begin{proof}
We first observe that we may assume $\nu\geq 2$. Indeed, for $\nu=1$ we simply extend $V_n$ from $\nu=3$ to an even function on $\R$. The proof below will show that the corresponding eigenfunction $\psi_n$ is radial and we can extend $r\psi_n$ to an odd function on $\R$ which will satisfy the correct equation.

Now let $\nu\geq 2$. We look for an eigenfunction of the form
$$
\psi(x) = \phi(r) w(r) \,, \qquad r=|x| \,,
$$
where $\phi$ is a radial function solving $-\Delta\phi=\phi$ in $\R^\nu$ (in particular, $\phi$ is regular at the origin). It is known that, up to a multiplicative constant, $\phi(r)=r^{-(\nu-2)/2} J_{(\nu-2)/2}(r)$, where $J_{(\nu-2)/2}$ is a Bessel function. This follows from Bessel's equation
$$
-J_{(\nu-2)/2}'' - r^{-1} J_{(\nu-2)/2}' + \left(\frac{\nu-2}2\right)^2 r^{-2} J_{(\nu-2)/2} = J_{(\nu-2)/2} \,,
$$
as well as
\begin{equation}
\label{eq:besselzero}
J_{(\nu-2)/2}(r) \sim \Gamma(\nu/2)^{-1} (r/2)^{(\nu-2)/2}
\qquad\text{as}\ r\to 0 \,.
\end{equation}
In the following we make use of the asymptotics
\begin{equation}
\label{eq:besselinfty}
J_{(\nu-2)/2}(r) = \sqrt{\frac{2}{\pi r}} \sin(r-\pi(\nu-3)/4)  + O(r^{-3/2})
\qquad\text{as}\ r\to\infty \,,
\end{equation}
which may also be differentiated with respect to $r$. (These asymptotics can be proved using Jost solutions, without referring to the theory of Bessel functions.) Using $-\Delta\phi=\phi$ we find
$$
-\Delta \psi = \psi - w' (2\phi'+ (\nu-1) r^{-1} \phi) - \phi w''
$$
with $(\cdot)'=\partial/\partial r$. Therefore, the eigenvalue equation for $\psi$ will be satisfied if we set
$$
V := \frac{w'}{w}\ \frac{2\phi'+ (\nu-1)r^{-1} \phi}{\phi} + \frac{w''}{w} \,.
$$
As usual, we want that $w'$ vanishes where $\phi$ vanishes and therefore we define
$$
g(r) := \int_0^r \phi(s)^2 s^{\nu-1} \,ds = \int_0^r J_{(\nu-2)/2}(s)^2 s\,ds
$$
The asymptotics \eqref{eq:besselinfty} show that
\begin{equation}
\label{eq:ginfty}
\lim_{r\to\infty} r^{-1} g(r) = \pi^{-1}
\end{equation}
We now define
$$
w_n(r) := (n^2 + g(r)^2)^{-\alpha}
$$
and we define $V_n$ with $w_n$ in place of $w$. As in the previous construction, the parameter $n$ need not be an integer, but we will use later that $n\geq 1$. Finally, we will choose $\alpha>\nu/4$, which by \eqref{eq:ginfty} will guarantee that $\psi\in L^2(\R^\nu)$. As before we do not keep track of how our estimates depend on $\alpha$.

A quick computation shows that
\begin{equation}
\label{eq:vn}
V_n = \frac{4 \alpha(\alpha+1)}{m_n^2} g^2 g'^2 - \frac{2\alpha}{m_n} \left( g'^2+gg''\right) - \frac{2\alpha}{m_n} g g' \frac{2\phi'+ (\nu-1)r^{-1}\phi}{\phi}
\end{equation}
with $m_n(r) := n^2 + g(r)^2$. We claim that we can bound
\begin{equation}
\label{eq:wvnbound}
|V_n| \leq C \left( m_n^{-1/2} + m_n^{-1} \right)
\end{equation}
with $C$ independent of $n$. Once this is shown we can use $n\geq 1$ to bound $m_n^{-1} \leq n^{-1} m_n^{-1/2}\leq m_n^{-1/2}$ and obtain $|V_n|\leq C m_n^{-1/2}$ which, in view of \eqref{eq:ginfty}, is equivalent to the bound stated in the theorem. Clearly this bound will imply $\|V_n\|_{L^p}\to 0$ if $p>\nu$.

Thus, it remains to prove \eqref{eq:wvnbound}. Using \eqref{eq:besselzero} and \eqref{eq:besselinfty} we obtain $g\leq m_n^{1/2}$ and $|g'|, |g''|\leq C$, which allows us to bound the first two terms on the right side of \eqref{eq:vn} by $C (m_n^{-1/2} + m_n^{-1})$. In order to bound the last term, we use $g' = \phi^2 r^{\nu-1}$, so
$$
g'\, \frac{2\phi'+ (\nu-1)r^{-1}\phi}{\phi} = r^{\nu-1} \phi (2\phi'+ (\nu-1)r^{-1}\phi) = ( r^{\nu-1}\phi^2)'
$$
Using again \eqref{eq:besselzero} and \eqref{eq:besselinfty} we obtain $|( r^{\nu-1}\phi^2)'|\leq C$, and therefore also the last term on the right side of \eqref{eq:vn} is bounded by $C m_n^{-1/2}$. This completes the proof of \eqref{eq:wvnbound} and of the theorem.
\end{proof}


\section{Bounds for $0\leq\gamma\leq 1/2$}\label{sec:fagain}

In this section we review the proofs in \cite{F} and show that these bounds are also valid for positive eigenvalues. Moreover, we shall prove bounds for potentials which belong to spaces of the form $L^{\gamma_1+\nu/2}+L^{\gamma_2+\nu/2}$.

Since we will use a similar argument later in Section \ref{sec:positive} we formulate the general principle in abstract terms.

\begin{proposition}\label{posevabstract}
Let $X$ be a separable complex Banach space of functions on $\R^\nu$ such that $L^2(\R^\nu)\cap X$ is dense in $X$ and such that the duality pairing $X^*\times X\to \C$ extends the inner product in $L^2(\R^\nu)$. Assume that
\begin{equation}
\label{eq:resboundabstract}
\|(-\Delta-z)^{-1}\|_{X\to X^*} \leq N(z) \,,
\end{equation}
where $N(z)$ is finite for $z\in\C\setminus[0,\infty)$ and continuous up to $[0,\infty)\setminus I$ for some set $I\subset [0,\infty)$. Assume that multiplication by $V:\R^\nu\to\C$ is a bounded operator from $X$ to $X^*$. Then, if $E\in\C\setminus I$ is an eigenvalue of $-\Delta+V$ in $L^2(\R^\nu)$ with an eigenfunction in $X^*$, then
$$
1\leq N(E)\, \|V\|_{X^*\to X} \,.
$$
\end{proposition}

\begin{proof}
We give the proof only for $E\in [0,\infty)\setminus I$, the case $E\in\C\setminus[0,\infty)$ being similar (and easier). We denote the eigenfunction by $\psi$ and observe that, since $\psi\in X^*$ and since multiplication by $V$ is bounded from $X^*$ to $X$,
\begin{equation}
\label{eq:posevabstractproof}
\|V\psi\|_X \leq \|V\|_{X^*\to X} \|\psi\|_{X*} \,,
\end{equation}
so $V\psi\in X$. Since $(-\Delta-E-i\epsilon)^{-1}$ is bounded from $X$ to $X^*$ and since, by the eigenvalue equation,
$$
\psi_\epsilon:=(-\Delta-E-i\epsilon)^{-1} (-\Delta-E)\psi = -(-\Delta-E-i\epsilon)^{-1} (V\psi) \,,
$$
we infer that $\psi_\epsilon\in X^*$ and
$$
\|\psi_\epsilon\|_{X^*} \leq N(E+i\epsilon) \, \|V\psi\|_X \,.
$$
Since $N(E+i\epsilon)\to N(E)$ as $\epsilon\to 0$, we see that the $\psi_\epsilon$ are uniformly bounded in $X^*$ and so they have a limit point in the weak-* topology of $X^*$. On the other hand, by dominated convergence in Fourier space, one easily verifies that $\psi_\epsilon\to\psi$ strongly (and hence also weakly) in $L^2(\R^\nu)$. Since $L^2(\R^\nu)\cap X$ is dense in $X$ and since the duality pairing $X^*\times X\to \C$ extends the inner product in $L^2(\R^\nu)$, we infer that the limit point in the weak-* topology of $X^*$ is unique and given by $\psi$. Moreover, by lower semi-continuity of the norm,
$$
\|\psi\|_{X^*} \leq \liminf_{\epsilon\to 0}\|\psi_\epsilon\|_{X^*} \leq \liminf_{\epsilon\to 0} N(E+i\epsilon)\, \|V\psi\|_X = N(E)\, \|V\psi\|_X 
$$
This, together with the bound \eqref{eq:posevabstractproof}, implies the bound in the proposition.
\end{proof}

Our first application of the abstract principle yields the following theorem, which extends the bound of \cite{F} to positive eigenvalues.

\begin{theorem}\label{posev}
Let $\nu\geq 2$, $0<\gamma\leq 1/2$ and $V\in L^{\gamma+\nu/2}(\R^\nu)$. Then any eigenvalue $E$ of $-\Delta+V$ in $L^2(\R^\nu)$ satisfies
$$
|E|^\gamma \leq C_{\gamma,\nu} \int_{\R^\nu} |V|^{\gamma+\nu/2}\,dx
$$
with $C_{\gamma,\nu}$ independent of $V$. Moreover, if $\nu\geq 3$ and
$$
\int_{\R^\nu} |V|^{\nu/2}\,dx < C_\nu \,,
$$
then $-\Delta+V$ in $L^2(\R^\nu)$ has no eigenvalue.
\end{theorem}

\begin{proof}
We apply Proposition \ref{posevabstract} with $X=L^p(\R^\nu)$, where $p$ is defined by $p/(2-p)=\gamma+\nu/2$, so that the assumptions on $\gamma$ become $2\nu/(\nu+2)<p\leq 2(\nu+1)/(\nu+3)$. Since $-\Delta+V$ is defined via $m$-sectorial forms, we know a-priori that an eigenfunction satisfies $\psi\in H^1(\R^\nu)$ and so, by Sobolev embedding theorems, $\psi\in L^{p'}(\R^\nu)=X^*$. Note also that, by H\"older's inequality,
$$
\|V\|_{X^*\to X} = \|V\|_{p/(2-p)}
$$
According to the Kenig--Ruiz--Sogge bound \eqref{eq:keruso} assumption \eqref{eq:resboundabstract} is satisfied with $N(z)= C|z|^{-\nu/2+\nu/p-1}$ and $I=\{0\}$. Therefore the claimed bound follows from Proposition \ref{posevabstract}. The second part of the theorem is proved similarly, taking $\gamma=0$, $I=\emptyset$ and noting that for $\nu\geq 3$ the bound \eqref{eq:keruso} holds also for $p=2\nu/(\nu+2)$. This completes the proof.
\end{proof}

\begin{remark}\label{noposev}
In a similar spirit we note that if $\nu=1$ and $V\in L^1(\R)$ (possibly complex-valued), then $-d^2/dx^2 + V(x)$ in $L^2(\R)$ has no positive eigenvalue. Thus the restriction that the bound $|E|^{1/2} \leq (1/2) \|V\|_1$ holds only for eigenvalues $E\in\C\setminus(0,\infty)$, which appears frequently in the literature, is unnecessary. (The absence of positive eigenvalues follows from standard Jost function techniques which show that for $k>0$ the equation $-\psi'' + V\psi = k^2 \psi$ has two solutions $\psi_+$ and $\psi_-$ with $\psi_\pm(x) \sim e^{\pm ikx}$ as $x\to\infty$, so no solution of this equation is square integrable. These arguments go back at least to Titchmarsh \cite{T}.)
\end{remark}

\begin{proposition}\label{lpsum}
Let $V_1\in L^{\gamma_1+\nu/2}(\R^\nu)$, $V_2\in L^{\gamma_2+\nu/2}(\R^\nu)$, where $0<\gamma_1<\gamma_2\leq 1/2$ if $\nu=2$ and $0\leq\gamma_1<\gamma_2\leq 1/2$ if $\nu\geq 3$. Then any eigenvalue $E\in\C\setminus\{0\}$ of $-\Delta+V_1+V_2$ in $L^2(\R^\nu)$ satisfies
$$
|E|^{-\gamma_1} \int_{\R^\nu} |V_1|^{\gamma_1+\nu/2} \,dx + |E|^{-\gamma_2} \int_{\R^\nu} |V|^{\gamma_2+\nu/2}\,dx \geq c_{\gamma_1,\gamma_2,\nu}>0 \,.
$$
\end{proposition}

\begin{proof}
Again we prove this only for positive eigenvalues, the other case being simpler. Let $\psi$ be the eigenfunction and let $\epsilon>0$ be a small parameter. We denote $S_\epsilon:=|-\Delta-E-i\epsilon| (-\Delta-E-i\epsilon)^{-1}$ and  $\phi_\epsilon:=|-\Delta-E-i\epsilon|^{1/2}\psi$, where $\psi$ is the eigenfunction. Since $\psi\in H^1(\R^\nu)$, $\phi_\epsilon\in L^2(\R^\nu)$. We can write the eigenvalue equation in the form
$$
S_\epsilon |-\Delta-E-i\epsilon|^{-1/2} V |-\Delta-E-i\epsilon|^{-1/2} \phi_\epsilon = -\frac{-\Delta-E}{-\Delta-E-i\epsilon} \phi_\epsilon \,.
$$
Therefore,
\begin{align}
\label{eq:lpsumproof}
\left\| \frac{-\Delta-E}{-\Delta-E-i\epsilon} \phi_\epsilon \right\| & = \| S_\epsilon |-\Delta-E-i\epsilon|^{-1/2} V |-\Delta-E-i\epsilon|^{-1/2} \phi_\epsilon \| \notag \\
& \leq \left( \left\| S_\epsilon |-\Delta-E-i\epsilon|^{-1/2} V_1 |-\Delta-E-i\epsilon|^{-1/2} \right\| \right. \notag \\
& \qquad \left. + \left\| S_\epsilon |-\Delta-E-i\epsilon|^{-1/2} V_2 |-\Delta-E-i\epsilon|^{-1/2}\right\| \right) \|\phi_\epsilon \| \,.
\end{align}
Since the operator norm of $AB$ equals that of $BA$, we have
$$
\left\| S_\epsilon |-\Delta-E-i\epsilon|^{-1/2} V_j |-\Delta-E-i\epsilon|^{-1/2} \right\| = \left\| (\sgn V_j) |V_j|^{1/2} (-\Delta-E-i\epsilon)^{-1} |V_j|^{1/2} \right\|
$$
and, as in \cite{F}, the Kenig--Ruiz--Sogge bound \eqref{eq:keruso} implies that
$$
\left\| (\sgn V_j) |V_j|^{1/2} (-\Delta-E-i\epsilon)^{-1} |V_j|^{1/2} \right\|
\leq C (|E|^2+\epsilon^2)^{-\gamma_j/(2\gamma_j+\nu)} \|V_j\|_{\gamma_j+\nu/2} \,.
$$
Inserting this into \eqref{eq:lpsumproof} we obtain
\begin{align}
\left\| \frac{-\Delta-E}{-\Delta-E-i\epsilon} \phi_\epsilon \right\|
& \leq C \left( (|E|^2+\epsilon^2)^{-\gamma_1/(2\gamma_1+\nu)} \|V_1\|_{\gamma_1+\nu/2} \right. \notag \\ 
&\qquad \left. + (|E|^2+\epsilon^2)^{-\gamma_2/(2\gamma_2+\nu)} \|V_2\|_{\gamma_2+\nu/2} \right) \|\phi_\epsilon\| \,.
\end{align}
Finally, we observe that $\|\phi_\epsilon\|\leq \|\phi\|<\infty$ and that $\frac{-\Delta-E}{-\Delta-E-i\epsilon} \phi_\epsilon \to \phi$ in $L^2(\R^\nu)$ (by dominated convergence in Fourier space. Thus, as $\epsilon\to 0$, we obtain the claimed bound in the theorem.
\end{proof}


\section{Bounds for $1/2<\gamma<\nu/2$}\label{sec:positive}

\subsection{Eigenvalue bounds}

In this section we show that \eqref{eq:keller} holds for $1/2<\gamma<\nu/2$ if $V$ is radial and, more generally, if for every $r>0$, $V(r\omega)$ is replaced by $\text{ess-sup}_{\omega\in\Sph^{\nu-1}} |V(r\omega)|$. The precise statement is

\begin{theorem}\label{evbound}
Let $\nu\geq 2$ and $1/2<\gamma<\nu/2$. Then
$$
|E|^\gamma \leq C_{\gamma,\nu} \int_0^\infty \|V(r\, \cdot)\|_{L^\infty(\Sph^{\nu-1})}^{\gamma+\nu/2} r^{\nu-1} \,dr \,.
$$
\end{theorem}

At the endpoint $\gamma=\nu/2$ we have the following bound

\begin{theorem}\label{endpoint}
Let $\nu\geq 2$. Then
$$
|E|^{\nu/2} \leq C_\nu \left( \int_0^\infty | \{ r>0: \emph{\text{ess-sup}}_{\omega\in\Sph^{\nu-1}} |V(r\omega)| > \tau \} |_\nu^{1/\nu} \,d\tau \right)^\nu \,,
$$
where $|\cdot|_\nu$ denotes the measure $|\Sph^{\nu-1}|\, r^{\nu-1}\,dr$ on $(0,\infty)$
\end{theorem}

Note that the integral on the right side in the theorem is the norm in the Lorentz space $L^{\nu,1}(\R_+,r^{\nu-1}\,dr; L^\infty(\Sph^{\nu-1}))$.

We will deduce Theorems \ref{evbound} and \ref{endpoint} from the following two resolvent bounds. The first one will imply Theorem \ref{evbound}.

\begin{theorem}\label{resmixed}
Let $\nu\geq 2$ and $2(\nu+1)/(\nu+3)< p<2\nu/(\nu+1)$. Then for all $f\in L^p(\R_+,r^{\nu-1}\,dr;L^2(\Sph^{\nu-1}))$ and $z\in\C\setminus[0,\infty)$,
\begin{align*}
& \left( \int_0^\infty \left( \int_{\Sph^{\nu-1}} |((-\Delta-z)^{-1} f)(r\omega)|^2 \,d\omega \right)^{p'/2} r^{\nu-1} \,dr \right)^{1/p'} \\
& \qquad \leq C_{p,\nu} |z|^{-\nu/2+\nu/p-1} \left( \int_0^\infty \left( \int_{\Sph^{\nu-1}} |f(r\omega)|^2 \,d\omega \right)^{p/2} r^{\nu-1} \,dr \right)^{1/p} \,.
\end{align*}
\end{theorem}

As explained in the introduction, we think of Theorem \ref{resmixed} as the analogue of the uniform Sobolev bounds by Kenig--Ruiz--Sogge \cite{KRS} which correspond to the range $2\nu/(\nu+2)<p\leq 2(\nu+1)/(\nu+3)$, see \eqref{eq:keruso}. Since uniform resolvent bounds imply Fourier restriction bounds (since $(-\Delta-\lambda-i\epsilon)^{-1} - (-\Delta-\lambda+i\epsilon)^{-1}\to 2\pi i \delta(-\Delta-\lambda)$ as $\epsilon\to 0+$), the Knapp counterexample \cite{St} shows that \eqref{eq:keruso} cannot hold for larger values of $p$. However, as we show, larger values of $p$ can be achieved by considering mixed norm spaces. The use of mixed norm spaces in the context of Fourier restriction bounds seems to have first appeared in Vega \cite{V}, who proved the corresponding restriction inequality in the range $2(\nu+1)/(\nu+3)< p<2\nu/(\nu+1)$ in dimensions $\nu\geq 3$; see also \cite{G} where $\nu=2$ is included as well. Our resolvent bound seems to be new, although our arguments follow closely those of Barcelo--Ruiz--Vega \cite{BRV}, and our assumption $p<2\nu/(\nu+1)$ is optimal, since the results of \cite{G} show that the corresponding Fourier restriction bound does not hold for $p\geq 2\nu/(\nu+1)$.

The following bound will imply Theorem \ref{endpoint}. As we will see, it is a rather straightforward consequence of the main result of \cite{BRV}.

\begin{theorem}\label{brvcor}
Let $\nu\geq 2$ and let $V$ be a non-negative, measurable function with
$$
\|V\|_{L^{\nu,1}(\R_+,r^{\nu-1}\,dr; L^\infty(\Sph^{\nu-1}))} = \int_0^\infty | \{ r>0: \emph{\text{ess-sup}}_{\omega\in\Sph^{\nu-1}} |V(r\omega)| > \tau \} |_\nu^{1/\nu} \,d\tau <\infty \,.
$$
Then, for all $f\in L^2(\R^\nu,V^{-1}\,dx)\cap L^2(\R^\nu)$ and $z\in\C\setminus[0,\infty)$,
$$
\int_{\R^\nu} |(-\Delta-z)^{-1}f|^2 V \,dx \leq C |z|^{-1} \|V\|_{L^{\nu,1}(\R_+,r^{\nu-1}\,dr; L^\infty(\Sph^{\nu-1}))}^2 \int_{\R^\nu} |f|^2 V^{-1} \,dx \,.
$$
\end{theorem}

Theorem \ref{evbound} follows from Theorem \ref{resmixed} by Proposition \ref{posevabstract} with the choice $X=L^p(\R_+,r^{\nu-1}\,dr;L^2(\Sph^{\nu-1}))$ in the same way as Theorem \ref{posev} was derived from \eqref{eq:keruso}. Similarly, Theorem \ref{endpoint} follows from Theorem \ref{brvcor} by Proposition \ref{posevabstract}; here we set $X=L^2(w^{-1})$ where $w=\max\{|V|,\delta G\}$, where $G$ is a strictly positive function in $L^{\nu,1}(\R_+,r^{\nu-1}\,dr; L^\infty(\Sph^{\nu-1}))$ (for instance, a Gaussian) and $\delta>0$ is a small parameter. Having $\delta>0$ implies that $L^2\cap L^2(w^{-1})$ is dense in $L^2(w^{-1})$. Moreover, one easily verifies that
$$
\|V\|_{L^2(w)\to L^2(w^{-1})} \leq 1 \,,
$$
so Proposition \ref{posevabstract} yields
$$
1\leq C |z|^{-1} \|\max\{|V|,\delta G\}\|_{L^{\nu,1}(\R_+,r^{\nu-1}\,dr; L^\infty(\Sph^{\nu-1}))}^2
$$
and as $\delta\to 0$ we obtain the claimed bound.

Thus, it remains to prove Theorems \ref{resmixed} and \ref{brvcor}.


\subsection{Proof of Theorem \ref{resmixed}}

It is well known that on spherical harmonics of degree $l\in\N_0$ the operator $-\Delta$ acts as
$$
h_l := -\partial_r^2 - (\nu-1)r^{-1}\partial_r + l(l+\nu-2) r^{-2} \,.
$$
This operator, with an appropriate boundary condition at the origin (coming from the decomposition into spherical harmonics), is self-adjoint in $L^2(\R_+, r^{\nu-1}\,dr)$. It is well-known that the boundary values of the resolvent $(h_l-\lambda-i0)^{-1}$ exist in suitably weighted spaces. The following proposition shows that these boundary values are bounded operators from $L^p(\R_+, r^{\nu-1}\,dr)$ to $L^{p'}(\R_+, r^{\nu-1}\,dr)$. The key observation is that their norms are bounded uniformly in $l\in\N_0$.

\begin{proposition}\label{resmixedradial}
For any $\nu\geq 2$ and $2\nu/(\nu+2)<p<2\nu/(\nu+1)$,
$$
\sup_{l\in\N_0} \left\| (h_l -1-i0)^{-1} \right\|_{L^p(\R_+,r^{\nu-1})\to L^{p'}(\R_+,r^{\nu-1})} <\infty \,.
$$
\end{proposition}

To prove this proposition we use the following simple criterion for the boundedness of an integral operator from $L^p$ to $L^{p'}$.

\begin{lemma}\label{intop}
Let $X$ and $Y$ be measure spaces and $k\in L^{p'}(X\times Y)$ for some $1\leq p\leq 2$. Then $(kf)(y) = \int_X k(x,y) f(x) \,dx$ defines a bounded operator from $L^p(X)$ to $L^{p'}(Y)$ with
$$
\|k\|_{L^p(X)\to L^{p'}(Y)} \leq \|k\|_{L^{p'}(X\times Y)} \,.
$$
\end{lemma}

\begin{proof}[Proof of Lemma \ref{intop}]
By Minkowski's and H\"older's inequality
\begin{align*}
\|kf\|_{p'}^{p'} & = \int_Y \left| \int_X k(x,y) f(x) \,dx \right|^{p'} \,dy \\
& \leq \left( \int_X \left( \int_Y |k(x,y)|^{p'} \,dy \right)^{1/p'} |f(x)| \,dx \right)^{p'} \\
& \leq \left( \int_X \int_Y |k(x,y)|^{p'} \,dy \,dx \right) \left( \int_X |f(x)|^p \,dx \right)^{p'/p} \,, 
\end{align*}
which yields the claimed inequality.
\end{proof}

Modulo a technical result about Bessel functions (Proposition \ref{bessel}), which we prove in the appendix, we now give the

\begin{proof}[Proof of Proposition \ref{resmixedradial}]
According to Sturm--Liouville theory $(h_l-1-i0)^{-1}$ is an integral operator with integral kernel
$$
(h_l-1-i0)^{-1}(r,r') = (rr')^{-(\nu-2)/2} J_{\mu_l}(\min\{r,r'\}) H_{\mu_l}^{(1)}(\max\{r,r'\}) \,,
$$
where $J_{\mu_l}$ and $H^{(1)}_{\mu_l}$ are Bessel and Hankel functions, respectively, and where $\mu_l=l+(\nu-2)/2$. Thus, by Lemma \ref{intop},
\begin{align*}
& \left\| (h_l-1-i0)^{-1} \right\|^{p'}_{L^p(\R_+,r^{\nu-1})\to L^{p'}(\R_+,r^{\nu-1})} \\ 
&\qquad \leq 2 \int_0^\infty \int_r^\infty |J_{\mu_l}(r)|^{p'} |H_{\mu_l}(r')|^{p'} (rr')^{-p'(\nu-2)/2 + \nu-1} \,dr'\,dr \,.
\end{align*}
The fact that the right side is finite and uniformly bounded in $l$ follows from Proposition \ref{bessel} in the appendix with $q=p'$. This completes the proof of the proposition.
\end{proof}

In order to deduce Theorem \ref{resmixed} from Proposition \ref{resmixedradial} we need the following general result. 

\begin{lemma}\label{mink}
Let $X$ and $Y$ be measure spaces and $1\leq p\leq 2$. Let $(K_j)$ be a sequence of bounded operators from $L^p(X)$ to $L^{p'}(Y)$. Let $\mathcal H$ be a separable Hilbert space with an orthonormal basis $(e_j)$ and define a linear operator $K$ by
$$
K (f\otimes e_j) = (K_jf) \otimes e_j \qquad
\text{for all}\ f\in L^p(X) \ \text{and all}\ j \,.
$$
Then $K$ is bounded from $L^p(X,\mathcal H)$ to $L^{p'}(Y,\mathcal H)$ with
$$
\| K \|_{L^p(X,\mathcal H)\to L^{p'}(Y,\mathcal H)} = \sup_j \|K_j\|_{L^p(X)\to L^{p'}(Y)} \,.
$$
\end{lemma}

\begin{proof}[Proof of Lemma \ref{mink}]
Since
\begin{align*}
\|K(f\otimes e_j)\|_{L^{p'}(Y,\mathcal H)} & = \|K_j f\|_{L^{p'}(Y)} 
\leq \|K_j\|_{L^p(X)\to L^{p'}(Y)} \|f\|_{L^{p}(X)} \\
& = \|K_j\|_{L^p(X)\to L^{p'}(Y)} \|f\otimes e_j\|_{L^{p}(Y,\mathcal H)} \,,
\end{align*}
we have $\|K\| \leq \sup \|K_j\|$ (with obvious indices). To prove the opposite bound we write $F=\sum f_j\otimes e_j$, so that
$$
\|KF\|_{L^{p'}(Y,\mathcal H)}^{p'} = \int_Y \left( \sum |(K_jf_j)(y)|^2 \right)^{p'/2} \,dy \,.
$$
Since $p'\geq 2$ we can bound this from above using Minkowski's inequality by
$$
\left( \sum \left( \int_Y |(K_jf_j)(y)|^{p'} \,dy \right)^{2/p'} \right)^{p'/2} \,,
$$
which in turn is bounded from above by
$$
\left( \sum \|K_j\|^2 \left( \int_X |f_j(x)|^{p} \,dx \right)^{2/p} \right)^{p'/2} 
\leq \left( \sup \|K_j\| \right)^{p'} \left( \sum \left( \int_X |f_j(x)|^{p} \,dx \right)^{2/p} \right)^{p'/2}  \,.
$$
Once again by Minkowski's inequality, using the fact that $p\leq 2$,
$$
\sum \left( \int_X |f_j(x)|^{p} \,dx \right)^{2/p} \leq \left( \int_X \left( \sum |f_j(x)|^2 \right)^{p/2} \,dx  \right)^{2/p} = \|F\|_{L^{p}(X,\mathcal H)}^2 \,.
$$
This proves that $\|KF\|_{L^{p'}(Y,\mathcal H)} \leq \left( \sup \|K_j\| \right) \|F\|_{L^{p}(X,\mathcal H)}$, as claimed.
\end{proof}

We are finally in position to give the

\begin{proof}[Proof of Theorem \ref{resmixed}]
Let $\nu\geq 2$ and $2(\nu+1)/(\nu+3)<p<2\nu/(\nu+1)$. (In fact, the proof works also for $2\nu/(\nu+2)<p\leq 2(\nu+1)/(\nu+3)$, but the inequality we obtain in that case is weaker than \eqref{eq:keruso}.) We begin with a well-known argument reducing the proof to the case $z=1$. For $f,g\in C_0^\infty(\R^\nu)$,
$$
z\mapsto z^{\nu/2-\nu/p+1} (g, (-\Delta-z)^{-1} f)
$$
is an analytic function in $\{\im z>0\}$, continuous up to the boundary, and satisfying
$$
|z|^{\nu/2-\nu/p+1} |(g, (-\Delta-z)^{-1} f)| \leq C_{r,\nu} |z|^{\alpha} \|f\|_r \|g\|_r
$$
for every $2\nu/(\nu+2)<r\leq 2(\nu+1)/(\nu+3)$ and a certain $\alpha$ depending on $r$. This follows from the Kenig--Ruiz--Sogge bound \eqref{eq:keruso}. Thus, by the Phragm\'en--Lindel\"of principle, 
$$
\sup_{\im z>0} |z|^{\nu/2-\nu/p+1} |(g, (-\Delta-z)^{-1} f)| = \sup_{\lambda\in\R} |\lambda|^{\nu/2-\nu/p+1} |(g, (-\Delta-\lambda-i0)^{-1} f)| \,.
$$
If we can show that the right side is bounded by $C_{p,\nu} \|f\|_{L^p(L^2)} \|g\|_{L^p(L^2)}$ (with the abbreviation $L^p(L^2)= L^p(\R_+,r^{\nu-1}\,dr;L^2(\Sph^{\nu-1}))$), then, by density, the bound will be valid for any $f,g\in L^p(L^2)$. Moreover, since 
$$
(g, (-\Delta-\overline z)^{-1} f) = ((-\Delta-z)^{-1} g, f) = \overline{(f,(-\Delta-z)^{-1} g)} \,,
$$
we will have shown the bound claimed in the theorem.

By scaling it suffices to prove the bound
\begin{equation}
\label{eq:resmixedproof}
|\lambda|^{\nu/2-\nu/p+1} |(g, (-\Delta-\lambda-i0)^{-1} f)| 
\leq C_{p,\nu} \|f\|_{L^p(L^2)} \|g\|_{L^p(L^2)}
\end{equation}
for $\lambda=\pm 1$ only. We begin with $\lambda=-1$. Since $(-\Delta+1)^{-1}$ is convolution with a function in $L^q$ for any $q<\nu/(\nu-2)$, Young's inequality yields
$$
|(g, (-\Delta-\lambda-i0)^{-1} f)| 
\leq C_{p,\nu}' \|f\|_{L^p} \|g\|_{L^p}
$$
for any $p>2\nu/(\nu+2)$. Since 
$$
\|f\|_{L^p} \leq |\Sph^{\nu-1}|^{(2-p)/2p} \ \|f\|_{L^p(L^2)}
$$
for $p\leq 2$, this bound for $\lambda=-1$ is stronger than what we shall prove for $\lambda=1$.

Therefore we have reduced the proof to showing \eqref{eq:resmixedproof} for $2(\nu+1)/(\nu+3)<p<2\nu/(\nu+1)$ and $\lambda=1$. This is the same as
$$
\|(-\Delta-1-i0)^{-1} f\|_{L^{p'}(L^2)} \leq C_{p,\nu} \|f\|_{L^p(L^2)} \,.
$$
To do so, we expand $f$ with respect to spherical harmonics $(Y_{l,m})$, with $l\in\N_0$ and $m$ running through a certain index set of cardinality depending on $l$,
$$
f(x) = \sum_{l,m} f_{l,m}(|x|) Y_{l,m}(x/|x|) \,,
$$
so that
$$
\int_0^\infty \left( \int_{\Sph^{\nu-1}} |f(r\omega)|^2 \,d\omega \right)^{p/2} r^{\nu-1} \,dr
= \int_0^\infty \left( \sum_{l,m} |f_{l,m}(r)|^2 \right)^{p/2} r^{\nu-1} \,dr \,.
$$
Separation of variables shows that
$$
\left( (-\Delta-1-i0)^{-1} f \right)(x) = \sum_{lm} \left( (h_l-1-i0)^{-1} f_{lm}\right)(|x|)\, Y_{lm}(x/|x|) \,,
$$
where $h_l$ was defined at the beginning of this subsection. By Lemma \ref{mink} we have
$$
\|(-\Delta-1-i0)^{-1}\|_{L^p(L^2)\to L^{p'}(L^2)} = \sup_{l\in\N_0} \| (h_l-1+i0)^{-1} \|_{L^p\to L^{p'}} \,.
$$
The right hand side is finite by Proposition \ref{resmixed}. This completes the proof of the theorem.
\end{proof}


\subsection{Proof of Theorem \ref{brvcor}}

We shall deduce Theorem \ref{brvcor} from the following theorem of Barcelo, Ruiz and Vega \cite{BRV}. They introduce the following norm,
$$
\|V\|_{MT} = \sup_{R>0} \int_R^\infty \frac{\text{ess-sup}_{\omega\in\Sph^{\nu-1}}|V(r\omega)|\, r}{(r^2-R^2)^{1/2}} \,dr <\infty \,.
$$

\begin{theorem}\label{brv}
Let $\nu\geq 2$ and let $V$ be a non-negative, measurable function with $\|V\|_{MT}<\infty$. Then, for all $f\in L^2(\R^\nu,V^{-1}\,dx)\cap L^2(\R^\nu)$ and $z\in\C\setminus[0,\infty)$,
$$
\int_{\R^\nu} |(-\Delta-z)^{-1}f|^2 V \,dx \leq C |z|^{-1} \|V\|_{MT}^2 \int_{\R^\nu} |f|^2 V^{-1} \,dx \,.
$$
\end{theorem}

Barcelo, Ruiz and Vega call $\|V\|_{MT}<\infty$ the `radial Mizohata--Takeuchi' condition, thus the subscript `MT'. They show that for radial $V$ this condition is, in fact, also necessary to have a bound of the form $\|u\|_{L^2(V)} \leq C|z|^{-1/2} \|(-\Delta-z)u\|_{L^2(V)}$.

\begin{proof}[Proof of Theorem \ref{brvcor}]
By Theorem \ref{brv} it suffices to show that for any $\nu\geq 2$,
\begin{equation}
\label{eq:brvcorproof}
\|V\|_{MT} \leq C_\nu \| V \|_{L^{\nu,1}(\R_+,r^{\nu-1},L^\infty(\Sph^{\nu-1}))} \,.
\end{equation}
Let $\rho_R(r) := r^{-\nu+2} (r^2-R^2)^{-1/2}\chi_{\{r>R\}}$. Then, by H\"older's inequality in Lorentz spaces, with $v(r):= \text{ess-sup}_{\omega\in\Sph^{\nu-1}}|V(r\omega)|$,
\begin{align*}
\int_R^\infty \frac{\text{ess-sup}_{\omega\in\Sph^{\nu-1}}|V(r\omega)|\, r}{(r^2-R^2)^{1/2}} \,dr
& = \int_0^\infty v(r) \rho_R(r) r^{\nu-1}\,dr \\
& \leq C \|v\|_{L^{\nu,1}(\R_+,r^{\nu-1})} \|\rho_R\|_{L^{\nu/(\nu-1),\infty}(\R_+,r^{\nu-1})} \\
& = C \| V \|_{L^{\nu,1}(\R_+,r^{\nu-1},L^\infty(\Sph^{\nu-1}))} \|\rho_1\|_{L^{\nu/(\nu-1),\infty}(\R_+,r^{\nu-1})}\,,
\end{align*}
where we used that, by scaling, $\|\rho_R\|_{L^{\nu/(\nu-1),\infty}(\R_+,r^{\nu-1})} = \|\rho_1\|_{L^{\nu/(\nu-1),\infty}(\R_+,r^{\nu-1})}$. One easily checks that $\rho_1 \in L^{\nu/(\nu-1),\infty}(\R_+,r^{\nu-1})$, which, after taking the supremeum over $R>0$, yields \eqref{eq:brvcorproof}.
\end{proof}

The next corollary contains further eigenvalue bounds which are consequences of Theorem \ref{brv}.

\begin{corollary}\label{evbrv}
Let $E\in\C$ be an eigenvalue of $-\Delta+V$ in $L^2(\R^\nu)$. Then
\begin{equation}
\label{eq:evbrv}
|E|^{1/2} \leq C_\nu \,\|V\|_{MT} \,.
\end{equation}
Moreover, for any $p\in(2,\infty]$,
\begin{equation}
\label{eq:evbrvp}
|E|^{1/2} \leq C_{p,\nu} \sum_{j\in\Z} \left( \int_{2^j}^{2^{j+1}} \|V(r\,\cdot)\|_{L^\infty(\Sph^{\nu-1})}^p r^{p-1} \,dr \right)^{1/p} \,.
\end{equation}
\end{corollary}

Clearly, \eqref{eq:evbrvp} for $p=\infty$ means
$$
|E|^{1/2} \leq C_\nu \sum_{j\in\Z} \left( \sup_{2^j<|x|<2^{j+1}} |x| |V(x)| \right) \,.
$$
Since $\sum_{j\in\Z} \left( \sup_{2^j<|x|<2^{j+1}} |x| (1+|x|)^{-1-\epsilon} \right)<\infty$ for $\epsilon>0$, this bound implies, in particular,
$$
|E|^{1/2} \leq C_{\nu,\epsilon} \ \text{ess-sup}_{x\in\R^\nu} (1+|x|)^{1+\epsilon} |V(x)| \,,
\qquad \epsilon>0 \,.
$$
which is the main result of \cite{S}.

\begin{proof}
Bound \eqref{eq:evbrv} follows from Theorem \ref{brv} by Proposition \ref{posevabstract} using the arguments after Theorem \ref{brvcor}. Having proved this, for \eqref{eq:evbrvp} it suffices to prove that
\begin{equation}
\label{eq:evbrvproof}
\|V\|_{MT} \leq C_{p,\nu} \sum_{j\in\Z} \left( \int_{2^j}^{2^{j+1}} \|V(r\,\cdot)\|_{L^\infty(\Sph^{nu-1})}^p r^{p-1} \,dr \right)^{1/p} \,.
\end{equation}
This bound is stated in \cite{BRV} without proof, so we include it for the sake of completeness. We abbreviate $v(r):= \|V(r\,\cdot)\|_{L^\infty(\Sph^{\nu-1})}$. Since $p>2$,
\begin{align*}
\int_R^{2R1} \frac{v(r) r}{(r^2-R^2)^{1/2}} \,dr
& \leq \left( \int_R^{2R} v(r)^p r^{p-1} \,dr \right)^{1/p} \left( \int_R^{2R} \left(\frac{r}{\sqrt{r^2-R^2}} \right)^{p'} \frac{dr}{r} \right)^{1/p'} \\
& = c_p \left( \int_R^{2R} v(r)^p r^{p-1} \,dr\right)^{1/p} \,.
\end{align*}
On the other hand, for $r\geq 2R$, $r/\sqrt{r^2-R^2} \leq 2/\sqrt 3$, and therefore
\begin{align*}
\int_{2R}^\infty \frac{v(r) r}{(r^2-R^2)^{1/2}} \,dr & \leq \frac{2}{\sqrt 3} \sum_{j=1}^\infty \int_{2^j R}^{2^{j+1}R} v(r) \,dr \\
& \leq \frac{2}{\sqrt 3} \sum_{j=1}^\infty \left( \int_{2^j R}^{2^{j+1}R} v(r)^p r^{p-1} \,dr \right)^{1/p} \left( \int_{2^j R}^{2^{j+1}R}  \frac{dr}{r} \right)^{1/p'} \\
& = \frac{2}{\sqrt 3} (\ln2)^{1/p'} \sum_{j=1}^\infty \left( \int_{2^j R}^{2^{j+1}R} v(r)^p r^{p-1} \,dr \right)^{1/p} \,.
\end{align*}
Picking $k\in\Z$ such that $2^k\leq R<2^{k+1}$ we easily deduce \eqref{eq:evbrvproof}.
\end{proof}



\appendix

\section{Bounds on Bessel functions}

The key ingredient in our proof of Proposition \ref{resmixedradial} was the following result about integrals of Bessel and Hankel functions.

\begin{proposition}\label{bessel}
Let $\nu\geq 2$ and $2\nu/(\nu-1)< q< 2\nu/(\nu-2)$. Then
$$
\sup_{\mu\geq 0} \int_0^\infty \int_r^\infty |J_\mu(r)|^q |H^{(1)}_\mu(r')|^q (rr')^{-q(\nu-2)/2 + \nu-1} dr\,dr' <\infty \,.
$$
\end{proposition}

We emphasize that in this result $\nu$ is not required to be integer and $\mu$ is not required to be a half-integer (although they will be in our application later on).

In this appendix we prove Proposition \ref{bessel} using the techniques of \cite{BRV}. Using WKB analysis, Barcelo, Ruiz and Vega prove the following uniform bounds on Bessel functions. We state their complete result although we will not use its full strength.

\begin{proposition}\label{besselbounds}
There is a constant $C>0$ and a constant $\alpha_0\in(0,1/2)$ such that the following holds for all $\mu\geq 1/2$.
\begin{enumerate}
\item For $0<r\leq 1$,
$$
|J_\mu(r)|\leq C \frac{(r/2)^{\mu}}{\Gamma(\mu+1)} \,,
\quad
|H^{(1)}_\mu(r)|\leq C \frac{\Gamma(\mu)}{(r/2)^{\mu}} \,.
$$
\item For $1\leq r\leq \mu\sech\alpha_0$,
$$
|J_\mu(r)|\leq C \frac{e^{-\mu\phi_\mu(r)}}{\mu^{1/2}} \,,
\quad
|H^{(1)}_\mu(r)|\leq C \frac{e^{\mu\phi_\mu(r)}}{\mu^{1/2}} \,.
$$
\item For $\mu\sech\alpha_0 \leq r\leq \mu-\mu^{1/3}$,
$$
|J_\mu(r)|\leq C \frac{e^{-\mu\phi_\mu(r)}}{\mu^{1/4}(\mu-r)^{1/4}} \,,
\quad
|H^{(1)}_\mu(r)|\leq C \frac{e^{\mu\phi_\mu(r)}}{\mu^{1/4}(\mu-r)^{1/4}} \,.
$$
\item For $\mu-\mu^{1/3}\leq r\leq \mu+\mu^{1/3}$,
$$
|J_\mu(r)|\leq C \frac{1}{\mu^{1/3}} \,,
\quad
|H^{(1)}_\mu(r)|\leq C \frac{1}{\mu^{1/3}} \,.
$$
\item For $r\geq \mu+\mu^{1/3}$,
$$
|J_\mu(r)|\leq C \frac{1}{r^{1/4}(r-\mu)^{1/4}} \,,
\quad
|H^{(1)}_\mu(r)|\leq C \frac{1}{r^{1/4}(r-\mu)^{1/4}} \,.
$$
\end{enumerate}
Here, the function $\phi_\mu$ is defined by $\phi_\mu(\mu\sech\alpha)=\alpha-\tanh\alpha$.
\end{proposition}

We split the proof of Proposition \ref{bessel} into two parts. The first part (which is analogous to Lemma 6 in \cite{BRV}) is

\begin{lemma}\label{besselint}
Let $q>0$ and $\rho>-1$ such that
$$
\frac{q}{2}>\rho+1 \,,
\quad
\frac{q}{3}\geq \rho +\frac{1}{3} \,.
$$
Then
$$
\sup_{\mu\geq 1/2} \left( \int_0^\infty |J_\mu(r)|^q r^\rho \,dr + \int_{\mu-\mu^{1/3}}^\infty |H^{(1)}_\mu(r)|^q r^\rho \,dr \right) <\infty \,. 
$$
\end{lemma}

Arguing slightly more carefully, we can replace the lower bound $\rho>-1$ by $\frac{q}{2} + \rho+1>0$. More generally, it can be improved to $\mu_0 q+ \rho+1>0$ if we restrict the supremum to $\mu\geq\mu_0\geq 1/2$. This is only needed to ensure the integrability of $|J_\mu(r)|^q r^\rho$ near $r=0$.

\begin{proof}[Proof of Lemma \ref{besselint}]
We are going to use the upper bounds from Proposition \ref{besselbounds}. Since they coincide for $J_\mu$ and $H^{(1)}_\mu$ in the range $r\geq \mu-\mu^{1/3}$, we only prove the lemma for $J_\mu$. We write $\int_0^\infty |J_\mu(r)|^q r^\rho \,dr = I_1+I_2+I_3+I_4+I_5+I_6$, where the different terms correspond to the following regions of integration:
\begin{align*}
I_1: \quad & 0<r\leq 1 \,,\\
I_2: \quad & 1<r\leq \mu\sech\alpha_0 \,,\\
I_3: \quad & \mu\sech\alpha_0< r\leq \mu-\mu^{1/3} \,,\\
I_4: \quad & \mu-\mu^{1/3}<r\leq \mu+\mu^{1/3} \,,\\
I_5: \quad & \mu+\mu^{1/3}<r\leq 2\mu \,,\\
I_6: \quad & r>2\mu \,.
\end{align*}
In each of the regions we use the bounds from Proposition \ref{besselbounds} and we only make a few remarks about the straightforward computations. The finiteness of $I_1$ requires $q\mu+\rho+1>0$, which follows from $\rho>-1$. To bound $I_2$ we use the fact that $|J_\mu(r)|\leq C\mu^{-1}$ for $0< r\leq \mu\sech\alpha_0$, which is an easy consequence of Proposition \ref{besselbounds}. To bound $I_3$ we split the region of integration into intervals $(\mu-2^{j+1}\mu^{1/3},\mu-2^j\mu^{1/3}]$ and use $\phi_\mu(r) \geq \phi_\mu(\mu- 2^j\mu^{1/3}) \geq C^{-1} \mu^{-1} 2^{3j/2}$ in each such interval. This yields $I_3 \leq C \mu^{-q/3+\rho+1/3}$, which is uniformly bounded in $\mu$ by assumption. We obtain the same bound on $I_4$ and, if $q>4$, on $I_5$. Finally, if $q/2-\rho-1>0$ then $I_6$ is finite and satisfies $I_6\leq C\mu^{-q/2+\rho+1}$. The same bound holds for $I_5$ if $q<4$ and, with a factor of $\ln\mu$, if $q=4$. This concludes the sketch of the proof.
\end{proof}

The second part in the proof of Proposition \ref{bessel} (which is analogous to equation (2.28) in \cite{BRV}) is

\begin{lemma}\label{besselintdouble}
Let $q>0$ and $\rho>-1$ such that
$$
\frac{q}{2}>\rho+1 \,,
\quad
\frac{q}{3}\geq \rho +\frac{1}{3} \,.
$$
Then
$$
\sup_{\mu\geq 1/2} \int_0^{\mu-\mu^{1/3}} \int_r^{\mu-\mu^{1/3}} |J_\mu(r)|^q |H^{(1)}_\mu(r')|^q (rr')^\rho \,dr'\,dr <\infty \,. 
$$
\end{lemma}

\begin{proof}[Proof of Lemma \ref{besselintdouble}]
We decompose the double integral as $I_1+I_2$, corresponding to the following regions of integration:
\begin{align*}
I_1: \quad & 0<r\leq \mu\sech\alpha_0 \,,\ r<r'\leq\mu-\mu^{1/3} \,,\\
I_2: \quad & \mu\sech\alpha_0< r\leq \mu-\mu^{1/3} \,, \ r<r'\leq\mu-\mu^{1/3} \,.
\end{align*}
To bound $I_1$ we use the fact that $r|H^{(1)}_\mu(r)|^2$ is a decreasing function of $r$ \cite[p. 446]{W} and obtain for $q/2>\rho+1$,
\begin{align*}
\int_r^{\mu-\mu^{1/3}} |H^{(1)}_\mu(r')|^q (r')^\rho \,dr' \leq r^{q/2} |H^{(1)}_\mu(r)|^q \int_r^\infty (r')^{\rho-q/2} \,dr' = \frac{r^{\rho+1}}{q/2-\rho-1} |H^{(1)}_\mu(r)|^q \,.
\end{align*}
The bounds from Proposition \ref{besselbounds} show that $|J_\mu(r)||H_\mu^{(1)}(r)| \leq C^2\mu^{-1}$ for $0<r\leq\mu\sech\alpha_0$, and therefore
$$
I_1 \leq \frac{C^{2q}\mu^{-q}}{q/2-\rho-1} \int_0^{\mu-\mu^{1/3}}  r^{2\rho+1} \,dr
\leq C' \mu^{-q+2\rho+2} \,.
$$
This is uniformly bounded since $q/2>\rho+1$.

To bound $I_2$ we argue similarly, but we estimate slightly differently
\begin{align*}
\int_r^{\mu-\mu^{1/3}} |H^{(1)}_\mu(r')|^q (r')^\rho \,dr' \leq r^{q/2} |H^{(1)}_\mu(r)|^q \int_r^\mu (r')^{\rho-q/2} \,dr' \leq r^{\rho}(\mu-r) |H^{(1)}_\mu(r)|^q \,.
\end{align*}
Proposition \ref{besselbounds} yields $|J_\mu(r)||H_\mu^{(1)}(r)| \leq C^2\mu^{-1/2}(\mu-r)^{-1/2}$ for $\mu\sech\alpha_0<r\leq\mu-\mu^{1/3}$, and therefore
$$
I_2 \leq C^{2q}\mu^{-q/2} \int_{\mu\sech\alpha_0}^{\mu-\mu^{1/3}}  (\mu-r)^{1-q/2} r^{2\rho} \,dr 
\leq C_q \mu^{2\rho-q/2} \int_{\mu\sech\alpha_0}^{\mu-\mu^{1/3}}  (\mu-r)^{1-q/2} \,dr \,.
$$
We conclude that
$$
I_2 \leq C_q' \times
\begin{cases}
\mu^{2\rho-2q/3+2/3} & \text{if}\ q>4 \,,\\
\mu^{2\rho-2} \ln\mu & \text{if}\ q=4 \,,\\
\mu^{2\rho-q+2} & \text{if}\ q<4 \,.
\end{cases}
$$
Under our assumptions on $q$ and $\rho$, this is uniformly bounded, as claimed.
\end{proof}

Finally, we give the

\begin{proof}[Proof of Proposition \ref{bessel}]
Let $\rho=-q(\nu-2)/2 + \nu-1$. The conditions $q<2\nu/(\nu-2)$ and $q>2\nu/(\nu-1)$ imply $\rho>-1$ and $q/2>\rho+1$, respectively. Finally, the condition $q/3\geq \rho+1/3$ follows from $q>2\nu/(\nu-1)$ and $\nu\geq 2$. Therefore we can apply Lemmas~\ref{besselint} and \ref{besselintdouble} and find that
\begin{align*}
& \int_0^\infty \int_r^\infty |J_\mu(r)|^q |H^{(1)}_\mu(r)|^q (rr')^{-q(\nu-2)/2 + \nu-1} dr\,dr' \\
& \qquad = \int_0^{\mu-\mu^{1/3}} \int_r^{\mu-\mu^{1/3}} |J_\mu(r)|^q |H^{(1)}_\mu(r)|^q (rr')^{-q(\nu-2)/2 + \nu-1} dr\,dr' \\
& \qquad \quad + \int_0^\infty \int_{\max\{r,\mu-\mu^{1/3}\}}^\infty |J_\mu(r)|^q |H^{(1)}_\mu(r)|^q (rr')^{-q(\nu-2)/2 + \nu-1} dr\,dr'
\end{align*}
is uniformly bounded in $\mu\geq 1/2$. The fact that the integrals are uniformly bounded for $0\leq\mu\leq1/2$ follows immediately from standard results about Bessel functions. This concludes the proof of the proposition.
\end{proof}



\bibliographystyle{amsalpha}

\end{document}